\documentclass[12pt,reqno,final]{amsart}

\usepackage{epsfig}
\usepackage{amsmath, amsthm, amsfonts}
\usepackage{graphicx}

\numberwithin{equation}{section}

\newcommand{\R}{{\mathbb R}}

\renewcommand{\Re}{{\operatorname{Re\,}}}

\newcommand{\bigk}{\mathop{\raisebox{-5pt}{\huge K}}}
\newcommand*{\fplus}{\genfrac{}{}{0pt}{}{}{+}}
\newcommand*{\fdots}{\genfrac{}{}{0pt}{}{}{\cdots}}
\newtheorem{theo}{{\sc \bf Theorem}}[section]
\newtheorem{cor}[theo]{{\sc \bf Corollary}}
\newtheorem{lem}[theo]{{\sc \bf Lemma}}

\newenvironment{example}{\noindent{\it Example:\/} }{}
\newenvironment{rem}{\noindent{\it {\bf Remark:}\/} }{}

\begin{document}

\title{A Value Region Problem For Continued Fractions and Discrete Dirac Equations}

\author{Slawomir Klimek}
\address{Department of Mathematical Sciences,
Indiana University-Purdue University Indianapolis,
402 N. Blackford St., Indianapolis, IN 46202, U.S.A.}
\email{sklimek@math.iupui.edu}

\author{Matt McBride}
\address{Department of Mathematics and Statistics,
Mississippi State University,
175 President's Cir., Mississippi State, MS 39762,  U.S.A.}
\email{mmcbride@math.msstate.edu }

\author{Sumedha Rathnayake}
\address{Department of Mathematics,
University of Michigan,
530 Church St., Ann Arbor, MI  48109, U.S.A.}
\email{sumedhar@umich.edu}

\author{Kaoru Sakai}
\address{Department of Mathematical Sciences,
Indiana University-Purdue University Indianapolis,
402 N. Blackford St., Indianapolis, IN 46202, U.S.A.}
\email{ksakai@iupui.edu}  
%\thanks{}

\date{\today}

\begin{abstract}
Motivated by applications in noncommutative geometry we prove several value range estimates for even convergents and tails, and odd reverse sequences of Stieltjes type continued fractions with bounded ratio of consecutive elements, and show how those estimates control growth of solutions of a system of discrete Dirac equations.  
\end{abstract}

\maketitle
\section{Introduction}
In this paper we prove useful growth estimates, see (\ref{dirac_cor}), for the solutions of a general class of discrete Dirac operators. This is achieved by translating such growth questions into a new type of value region problem for continued fractions, and then solving it by careful analysis of certain types of M\o bius transformations. This investigation was motivated by our work in \cite{KM1}, \cite{KM2}, \cite{KM3} on Dirac operators and their quantum analogs, subject to global boundary conditions, on various classical and noncommutative domains. 

To explain the particular problem we consider in this paper, we start with the following Dirac operator on $\R$:

\begin{equation*}
D = \left(
\begin{array}{cc}
-m & \frac{d}{dx} \\
-\frac{d}{dx} & m
\end{array}\right),
\end{equation*}
for $m\in\R$.  This operator $D$ is a simplified, one-dimensional version of the Dirac operator studied in \cite{KM3}.  The corresponding system of Dirac equations:

\begin{equation}\label{dirac_equation}
\left\{
\begin{aligned}
-mg + \frac{dh}{dx} &= 0 \\
-\frac{dg}{dx} + mh &= 0 \\
g(0) = 0,& \ h(0) = 1
\end{aligned}\right.
\end{equation}
can be easily solved, and solutions are $g(x) = \sinh(mx)$ and $h(x) = \cosh(mx)$.  Since the ratio of $g$ and $h$ is the hyperbolic tangent function, the ratio stays bounded for all $x\in\R$, or equivalently for all $m\in\R$.  

In the two-dimensional situation investigated in \cite{KM3}, the solutions of the similar system were comprised of the modified Bessel functions of the second kind, and the knowledge of the behavior of those solutions for large $|m|$ was instrumental in establishing compactness of the resolvent of the corresponding Dirac operator.
   
In \cite{KM4}, a noncommutative analog of the Dirac operator of \cite{KM3} was introduced and studied. It turned out that this operator is essentially a form of  discretization of the classical Dirac operator.  For our simplified $D$, the discrete analog of system (\ref{dirac_equation}), implied by noncommutative geometry, is:

\begin{equation*}
\left\{
\begin{aligned}
-mg_n + \frac{h_{n+1} - h_n}{\alpha_n} &=0 \\
\frac{g_n - g_{n+1}}{\beta_n} + mh_{n+1} &=0
\end{aligned}\right.
\end{equation*}
for some sequences $\alpha_n$ and $\beta_n$ of positive numbers and $m>0$.   
Exactly such systems and the corresponding operators were studied from the spectral theory point of view in \cite{AB}, \cite{A}, \cite{BO}, and references therein, so it would be interesting to investigate how the results in this paper can be used in spectral analysis.

If we rewrite the above equation in a matrix form we arrive at the following:

\begin{equation}\label{discrete_dirac_matrix_equation}
\left(
\begin{array}{c}
g_{n+1} \\ h_{n+1}
\end{array}\right) = \left(
\begin{array}{cc}
1 + m^2\alpha_n\beta_n & m\beta_n \\
m\alpha_n & 1
\end{array}\right)\left(
\begin{array}{c}
g_n \\ h_n
\end{array}\right)
\end{equation}
The key point is, as indicated in the next section, that this is precisely the equation satisfied by numerators and denominators of the convergents of continued fractions of the form $\bigk_{k=1}^\infty\left(\frac{1}{b_k}\right)$ with $b_{2k}=m\alpha_k$ and $b_{2k+1}=m\beta_k$, see equations (\ref{H-equations}). Consequently, the ratios of $g_n$ and $h_n$, for particular initial conditions, can be expressed as continued fractions.  The problem we are interested in is to determine if those continued fractions stay bounded for large $m$. The answer to this problem, see Corollary \ref{dirac_cor}, turned out to be essential for establishing compactness of the resolvent of the Dirac operator in \cite{KM4}.

In fact in this paper we look at a far more general situation than needed for \cite{KM4}; we consider continued fractions whose elements are complex numbers with positive real part (Stieltjes type), have bounded ratio of consecutive elements,  and we look at a full ``value region problem" (see  \cite{WA}, Chapter VIII).  Namely, we ask about a (bounded) region in the complex plane which contains values of the convergents of a continued fraction of our special type. More specifically, in our value region problem we consider only even convergents for continued fractions of the Stieltjes type with bounded ratio of consecutive elements, whether the fractions are convergent or not. Even more generally, we also study the same question for even tail sequences and for what we call odd reverse sequences associated with a continued fraction. The main results we obtain in this paper show that two types of circles with sufficiently large radii form such value regions. In fact, we obtain that circles centered at the origin and with sufficiently large radii are value regions for Van Vleck fractions with narrow sector angle $\theta<\pi/4$, assuming the ratio of two consecutive elements is bounded. We also provide an example illustrating that the results cannot be extended to sectors with the angle $\theta\geq\pi/4$. However, as described in our second main result, a shifted circle through the origin, centered on the real axis, and large enough radius is a value region for even convergents for a much larger class of continued fractions including all Van Vleck fractions with bounded ratio of two consecutive elements. 

While the considerations in this paper are fairly elementary, there are very few results in the literature on continued fractions that establish boundedness of continued fractions when all elements are scaled. One such result  can be deduced from Stieltjes' paper \cite{ST}. For $b_n>0$ and for $w$ such that $|\text{Arg}(w)|\leq\theta<\pi/2$, Stieltjes proved that there are bounded, nondecreasing functions $\phi_{2n}(t)$ on $[0,\infty)$ such that
\begin{equation*}
\left|\frac{1}{wb_1}\fplus\frac{1}{wb_2}\fplus\frac{1}{wb_3}\fplus\fdots\fplus\frac{1}{wb_{2n-1}}\fplus\frac{1}{wb_{2n}}\right|\leq\frac{1}{2\cos\theta}\int_0^\infty\frac{d\phi_{2n}(t)}{\sqrt t},
\end{equation*}  
i.e. the even convergents stay bounded when scaling all elements by the parameter $w$ in a sector.  However the functions $\phi_{2n}(t)$ in the formula above are not easily interpretable in term of the coefficients $b_n$. Moreover similar considerations do not seem to work for tails and reverse sequences needed for \cite{KM4}.
 
Another similar result called Lima\c{c}on Theorem is contained in \cite{L2}. However it assumes that $b_n\geq b>0$ which, for the purpose of applications in \cite{KM4}, is too restrictive.

Finally the inequality in \cite{KM4}:
\begin{equation*}
\frac{1}{wb_1}\fplus\frac{1}{wb_2}\fplus\frac{1}{wb_3}\fplus\fdots\fplus\frac{1}{wb_{2n-1}}\fplus\frac{1}{wb_{2n}}\leq
\sum_{i=i}^{n}\frac{wb_{2i-1}}{1+w^2b_{2i-1}b_{2i}},
\end{equation*}  
for positive $b_i$ and $w$, requires extra convergence assumptions to be usable in the limit $n\to\infty$.

The material in this paper is divided into two parts. Section \ref{sec2} contains our notation and the statements of the results while all the proofs are deferred to Section \ref{sec3}.

%%%%%%%%%%%%%%%%%%%%%%%%%%%%%%%%%%%%%%%%%%%%%%%%%%%%%%%%%%%%%%%%%%%%%%%%%%%%%%%%%%%%%%%%%%%%%%%%%%%%%%%%%%%%%%%%%%%%%%%%%%%%%%%%%%%%%%%%%%%%%%%%%%%%%%%%%%%%%%%%%%%%%%%%%%%%%%%%%

\section{Notation and Results}\label{sec2}

Given a sequence of nonzero complex numbers $\{b_n\}$, we associate to it a continued fraction of the form 

\begin{equation*}
\bigk_{k=1}^\infty\left(\frac{1}{b_k}\right)=\cfrac{1}{b_1+ \cfrac{1}{b_2+ \cfrac{1}{b_3 + \cdots}}} = \frac{1}{b_1}\fplus\frac{1}{b_2}\fplus\frac{1}{b_3}\fplus\fdots\fplus\frac{1}{b_n}\fplus\fdots
\end{equation*} 

The sequence of convergents $\{f_n\}=\{\frac {A_n}{B_n} \}$ (also referred to as the $n^{th}$ approximant) is defined by $ f_n=\frac{A_n}{B_n}=\bigk_{k=1}^n\left(\frac{1}{b_k}\right)$. The terms $A_n$ and $B_n$ are known to satisfy the Wallis-Euler recurrence relation, see for example \cite{CPVWJ}:
\begin{equation}\label{wallis_euler_recurrence_rel}
\begin{aligned}
A_n &= b_nA_{n-1}+A_{n-2} \ \textrm{ and }\\
B_n &= b_nB_{n-1}+B_{n-2}  \ \textrm{ for } n\geq 2 \\
\textrm{ with } A_0 &=0,\: A_1=1,\: B_0=1,\: B_1=b_1.
\end{aligned}
\end{equation} 

They also satisfy the following determinant relations.
\begin{equation*}
\begin{aligned}
A_nB_{n-1}-A_{n-1}B_n &=\left(-1\right)^{n-1}, \ n\geq 0 \\
A_nB_{n-2}-A_{n-2}B_n &=\left(-1\right)^n b_n, \ n\geq 1.
\end{aligned}
\end{equation*}

Using equations (\ref{wallis_euler_recurrence_rel}) we can easily derive the following matrix equation:

\begin{equation}\label{H-equations}
\left(
\begin{array}{c}
H_{2n+1} \\ H_{2n}
\end{array}\right) = \left(
\begin{array}{cc}
1 + b_{2n}b_{2n+1} & b_{2n+1} \\
b_{2n} & 1
\end{array}\right)\left(
\begin{array}{c}
H_{2n-1} \\ H_{2n-2}
\end{array}\right)
\end{equation}
where  $H_n$ is either  $A_n$ or $B_n$. By comparing with equation (\ref{discrete_dirac_matrix_equation}) we see that those are precisely the equation for the solutions of discrete Dirac operators.

Consider the M\"obius transformation $s_n : \mathbb C \cup \{\infty \} \to \mathbb C \cup \{\infty \}$ defined by $ s_n(w)=\frac{1}{b_n+w}$. It is easily seen by induction that 
\begin{equation}\label{convergent_formula}
s_1\circ s_2\circ\cdots\circ s_n(w)=\frac{A_n+wA_{n-1}}{B_n+wB_{n-1}}=: f_n(w).
\end{equation}
Notice that $s_1\circ s_2\circ\cdots\circ s_n(0)=f_n(0)=\frac{A_n}{B_n}=f_n$. From now on we will write $f_n$ for $f_n(0)$.

Given  a continued fraction $\bigk_{k=1}^\infty\left(\frac{1}{b_k}\right)$ we define a tail sequence  $\{t_n(w)\}$ starting with a constant $w\in \mathbb C \cup \{\infty \}$, see \cite{CPVWJ}, recursively by 
\begin{equation*}
t_{n-1}(w)=s_n(t_n(w))=\frac1{b_n+t_n(w)} 
\end{equation*}
and $t_0(w)=w$. From equation \eqref{wallis_euler_recurrence_rel} we see that $-\frac{B_{n-1}}{B_{n-2}}=\frac 1{b_n-\frac{B_n}{B_{n-1}}}$. This gives an example of a tail sequence $t_n(\infty)=-\frac {B_n}{B_{n-1}}$. Similarly, $t_n(0)=-\frac{A_n}{A_{n-1}}$ gives another example of a tail sequence. Next we define a reverse sequence $\{r_n(w)\}$ starting with $w$ recursively by 
\begin{equation*}
r_{n+1}(w)=s_n(r_n(w))=\frac 1{b_n+r_n(w)}
\end{equation*}
where $r_1(w)=w$ is an arbitrary number in the extended complex plane. Two such examples of  reverse sequences are $r_{n+1}(0)=\frac{B_n-1}{B_n}$ and $r_{n+1}(\infty)=\frac{A_{n-1}}{A_n}$. Similarly to \eqref{convergent_formula}, one can easily verify by induction that the tail and reverse sequences satisfy the following relations:

\begin{equation}\label{tail_and_reverse_formula}
t_{n}(w)= \frac{A_{n}-wB_{n}}{-A_{n-1}+wB_{n-1}}\ \textrm{ and } \ r_{n+1}(w)=\frac{B_{n-1}+wA_{n-1}}{B_n+wA_n}=s_n\circ s_{n-1}\circ\cdots\circ s_1(w),
\end{equation}
the last equality justifying the name: reverse sequences. It also follows that $t_{n}(w)$ is the inverse function of $f_{n}(w)$ and $r_{n+1}(w)=-1/t_{n}(-1/w)$.

Below we introduce three major types of continued fractions that we encounter in this paper. 

\begin{itemize}
\item
A  continued fraction of the form $\textrm{\large K}_{n=1}^\infty\big(\frac{1}{b_n}\big)$ with $\sum_{n=1}^{\infty} |b_n|< \infty$  will be called a {\bf Stern-Stolz fraction}. 
\item
If there is  $\theta$ with $0\leq \theta < \pi/2$ such that $b_n \in S_{\theta}=\{z:|\textrm{Arg}\,z|\leq\theta\}$ for all $n$ then $\textrm{\large K}_{n=1}^\infty\big(\frac{1}{b_n}\big)$ will be called a {\bf Van Vleck fraction}. 
\item
A continued fraction of the form $\textrm{\large K}_{n=1}^\infty\big(\frac{1}{b_n }\big)$ where $b_n$'s are complex numbers with positive real parts will be called a {\bf Stieltjes type fraction}. In particular a Van Vleck fraction is a Stieltjes type fraction.

\end{itemize}

The following classical theorem regarding the convergence of a Stern-Stolz fraction states that the limits of numerators and denominators of the even and odd convergents exist even if the continued fraction itself is not convergent. As a consequence, the even and odd convergents  $f_{2n}, f_{2n+1}$  also converge to limits denoted by $f_{\textrm{even}}$ and $f_{\textrm{odd}}$ respectively. 

\begin{theo}[Stern-Stolz]
If $\ \textnormal{\large K}_{n=1}^\infty\big(\frac{1}{b_n}\big)$ is a Stern-Stolz fraction then the sequences $\{A_{2n}\}$, $\{A_{2n+1}\}, \{B_{2n}\}$ and $\{B_{2n+1}\}$ converge to the limits $A_{\textrm{even}}, A_{\textrm{odd}}, B_{\textrm{even}}$ and $B_{\textrm{odd}}$ respectively. Moreover, $A_{\textrm{odd}}B_{\textrm{even}}-A_{\textrm{even}}B_{\textrm{odd}}=1$. Therefore $f_{\textrm{even}}-f_{\textrm{odd}} \neq 0$ and hence the continued fraction is not convergent.
\end{theo}
For more details on the theorem see \cite{BS}.

As a consequence of the above theorem and the formula (\ref{tail_and_reverse_formula}), we have that for Stern-Stolz fractions \  $t_{2n}(w), t_{2n+1}(w), r_{2n}(w)$ and $r_{2n+1}(w)$ converge to finite limits $t_{\textrm{even}}(w)$, $t_{\textrm{odd}}(w)$, $r_{\textrm{even}}(w)$ and $r_{\textrm{odd}}(w)$ respectively. Moreover, $t_{\textrm{even}}(w)=\frac1{t_{\textrm{odd}}(w)}$ and $r_{\textrm{even}}(w)=\frac1{r_{\textrm{odd}}(w)}$.

The classical Van Vleck theorem below gives the necessary and sufficient condition for the convergence of a Van Vleck continued fraction.

\begin{theo}[Van Vleck]\label{Van_Vleck}
A Van Vleck fraction $\textnormal{\large K}_{n=1}^\infty\big(\frac{1}{b_n}\big)$ is convergent if and only if $\sum_{n=1}^{\infty}|b_n|=\infty$.
\end{theo}
Several different proofs of this theorem are contained in \cite{BS}, \cite{GW}, \cite{L1}.

\rem It follows that for Van Vleck fractions $f_{\textrm{even}}(0)$ and $f_{\textrm{odd}}(0)$ always exist. If a Van Vleck fraction
$\textnormal{\large K}_{n=1}^\infty\big(\frac{1}{b_n}\big)$ is convergent then $f_{\textrm{even}}(0)=f_{\textrm{odd}}(0)=:f(0)$.

The first of our main theorems in this note is the following estimate for even  Van Vleck continued fractions with bounded ratio of consecutive elements in the sector $S_{\theta}$ for $\theta < \pi/4$. It shows that under those conditions the fractions stay inside an origin centered disk of sufficiently large radius.

\begin{theo}\label{main_theorem}
Suppose $C$ is a constant such that \ $C^2 \geq \sup\limits_{1\leq k\leq n} \left|\frac {p_k}{q_k}\right|\cdot \frac 1{\cos 2\theta}$ \ where $p_k, q_k, z \in S_{\theta}$ with $0\leq \theta < \pi/4$ and $p_k, q_k \neq 0$. If $|z|\leq C$ then
\begin{equation*}
\left|\frac{1}{q_1}\fplus\frac{1}{p_1}\fplus\frac{1}{q_2}\fplus\frac{1}{p_2}\fplus\fdots\fplus\frac{1}{q_n}\fplus\frac{1}{p_n+z}\right| \leq C.
\end{equation*}
\end{theo}

Now we look at the consequences of this estimate. Since 
$$f_{2n}(w)=\frac{1}{b_1}\fplus\frac{1}{b_2}\fplus\frac{1}{b_3}\fplus\fdots\fplus\frac{1}{b_{2n-1}}\fplus\frac{1}{b_{2n}+w},$$
$$ r_{2n+1}(w)=\frac{1}{b_{2n}}\fplus\frac{1}{b_{2n-1}}\fplus\frac{1}{b_{2n-2}}\fplus\fdots\fplus\frac{1}{b_2}\fplus\frac{1}{b_1+w},$$
$$t_{2n}(w)=\frac{1}{b_{2n+1}}\fplus\frac{1}{b_{2n+2}}\fplus\frac{1}{b_{2n+3}}\fplus\fdots\fplus\frac{1}{b_{2N-1}}\fplus\frac{1}{b_{2N}+t_{2N}(w)},$$
we immediadely obtain the following information about the covergents, reverse and tail sequences.

\begin{cor}\label{main_lemma1}
Let $\textnormal{\large K}_{n=1}^\infty\big(\frac{1}{b_n}\big)$ be a Van Vleck fraction with $b_n \in S_{\theta}$ where $0\leq \theta < \pi/4$. Assume that \  $\sup\limits_n \left|\frac{b_{2n}}{b_{2n-1}}\right| < \infty$ and let $C^2 \geq \sup\limits_n \left|\frac{b_{2n}}{b_{2n-1}}\right|\cdot \frac1{\cos 2\theta}$. If $w\in S_{\theta}$ and $|w|\leq C$ then $f_{2n}(w) \in S_{\theta}$ and $|f_{2n}(w)|\leq C$.
\end{cor}

A parallel result also holds for odd reverse sequences.

\begin{cor}\label{main_theorem_i}
Let $\textnormal{\large K}_{n=1}^\infty\big(\frac{1}{b_n}\big)$ be a Van Vleck fraction with $b_n \in S_{\theta}$ where $0\leq \theta < \pi/4$. Assume that \  $\sup\limits_n \left|\frac{b_{2n-1}}{b_{2n}}\right| < \infty$ and let $C^2 \geq \sup\limits_n \left|\frac{b_{2n-1}}{b_{2n}}\right|\cdot \frac1{\cos 2\theta}$. Then the following holds:
 if $w\in S_{\theta}$ and $|w|\leq C$ then $r_{2n+1}(w) \in S_{\theta}$ and $|r_{2n+1}(w)|\leq C$.
\end{cor}

Even tail sequences behave differently, as the later sequence terms determine the size of the previous terms as described in the following.

\begin{cor}\label{main_theorem_t}
Let $\textnormal{\large K}_{n=1}^\infty\big(\frac{1}{b_n}\big)$ be a Van Vleck fraction with $b_n \in S_{\theta}$ where $0\leq \theta < \pi/4$. Assume that \  $\sup\limits_n \left|\frac{b_{2n}}{b_{2n-1}}\right| < \infty$ and let $C^2 \geq \sup\limits_n \left|\frac{b_{2n}}{b_{2n-1}}\right|\cdot \frac1{\cos 2\theta}$. If \, $t_{2N}(w)\in S_{\theta}$ and $|t_{2N}(w)|\leq C$ then $t_{2n}(w)\in S_{\theta}$ and $|t_{2n}(w)|\leq C$ for all $n\leq N$.
\end{cor}

From the results above we can trivially, by taking the limit, deduce the following corollary.

\begin{cor}\label{corollary_to_main_theorem}
Let $\textnormal{\large K}_{n=1}^\infty\big(\frac{1}{b_n}\big)$ be a Van Vleck fraction with $b_n \in S_{\theta}$ where $0\leq \theta < \pi/4$.
Then the following results hold:

\begin{enumerate}
\item Assume that \  $\sup\limits_n \left|\frac{b_{2n}}{b_{2n-1}}\right| < \infty$ and let $C^2 \geq \sup\limits_n \left|\frac{b_{2n}}{b_{2n-1}}\right|\cdot \frac1{\cos 2\theta}$. If $w\in S_{\theta}$,
$|w|\leq C$ then $f_{\textrm{even}}(w) \in S_{\theta}$ and $|f_{\textrm{even}}(w)|\leq C$.
\item If $\sum_{n=1}^{\infty} |b_n|= \infty$, i.e. the fraction is convergent, then $|f(0)|^2\leq \sup\limits_n \left|\frac{b_{2n}}{b_{2n-1}}\right|\cdot \frac1{\cos 2\theta}$. 
\item Assume that $\sum_{n=1}^{\infty} |b_n|<\infty$ and  $\sup\limits_n \left|\frac{b_{2n-1}}{b_{2n}}\right| < \infty$, and let $C^2 \geq \sup\limits_n \left|\frac{b_{2n-1}}{b_{2n}}\right|\cdot \frac1{\cos 2\theta}$. If $w\in S_{\theta}$, $|w|\leq C$ then $r_{\textrm{odd}}(w) \in S_{\theta}$ and  $|r_{\textrm{odd}}(w)|\leq C$.
\item Assume that $\sum_{n=1}^{\infty} |b_n|<\infty$ and $\sup\limits_n \left|\frac{b_{2n}}{b_{2n-1}}\right| < \infty$, and let $C^2 \geq \sup\limits_n \left|\frac{b_{2n}}{b_{2n-1}}\right|\cdot \frac1{\cos 2\theta}$. If $t_{{\textrm{even}}}(w)\in S_{\theta}$ and $|t_{{\textrm{even}}}(w)|\leq C$ then $t_{2n}(w) \in S_{\theta}$ and $|t_{2n}(w)|\leq C$ for all $n$.
\end{enumerate}
\end{cor}

Next we will describe similar results but for different circles as value regions, namely circles through the origin. The results are stronger in this case and can be applied to Stieltjes type fractions as well as to general Van Vleck fractions with coefficients in arbitrary sectors.

\begin{theo}\label{main_lemma2}
Suppose $C$ is a constant such that $ C^2 \geq \frac14 \sup\limits_{1\leq k\leq n} \frac 1{\Re (q_{k}) \Re {\left(\frac 1{p_{k}}\right)}}$ where $p_k, q_k, z \in S_{\theta}$ with $0\leq \theta < \pi/2$ and $p_k, q_k \neq 0$. If $|z-C|\leq C$ then
\begin{equation*}
\left|\frac{1}{q_1}\fplus\frac{1}{p_1}\fplus\frac{1}{q_2}\fplus\frac{1}{p_2}\fplus\fdots\fplus\frac{1}{q_n}\fplus\frac{1}{p_n+z}-C\right| \leq C.
\end{equation*}
\end{theo}

As before we have the following list of consequences.

\begin{cor}\label{main_theorem1}
Let $\textnormal{\large K}_{n=1}^\infty\big(\frac{1}{b_n}\big)$ be a Stieltjes type continued fraction. Let $C$ be a constant such that $ C^2 \geq \frac14 \sup\limits_n \frac 1{\Re (b_{2n-1}) \Re {\left(\frac 1{b_{2n}}\right)}}$, assuming that  $\sup\limits_n \frac 1{\Re(b_{2n-1})\Re {\left(\frac 1{b_{2n}}\right)}} < \infty$.  Then the following results hold:

\begin{enumerate}
\item If $|w-C|\leq C$ then $|f_{2n}(w)-C|\leq C$.
\item If $\sum_{n=1}^{\infty} |b_n|< \infty$ then $|f_{\textrm{even}}(w)-C|\leq C$ whenever $|w-C|\leq C$.
\end{enumerate}

\end{cor}

Similarly for odd reverse sequences we have the following results.

\begin{cor}\label{main_theorem2}
Let $\textnormal{\large K}_{n=1}^\infty\big(\frac{1}{b_n}\big)$ be a Stieltjes type fraction with \,$\sup\limits_n \frac 1{\Re (b_{2n})\Re {\left(\frac 1{b_{2n-1}}\right)}} < \infty$. Let $ C^2 \geq \frac14 \sup\limits_n \frac 1{\Re (b_{2n}) \Re {\frac 1{b_{2n-1}}}}$. Then the following is true.

\begin{enumerate}
\item If $|w-C|\leq C $ then $|r_{2n+1}(w)-C|\leq C$.
\item If $\sum_{n=1}^{\infty} |b_n|< \infty$ then $|r_{\textrm{odd}}(w)-C|\leq C$ whenever $|w-C|\leq C$.
\end{enumerate}

\end{cor}

As before, the even tail sequences behave a bit differently.

\begin{cor}\label{main_theorem3}
Let $\textnormal{\large K}_{n=1}^\infty\big(\frac{1}{b_n}\big)$ be a Stieltjes type continued fraction. Let $C$ be a constant such that $ C^2 \geq \frac14 \sup\limits_n \frac 1{\Re (b_{2n-1}) \Re {\left(\frac 1{b_{2n}}\right)}}$, assuming that  $\sup\limits_n \frac 1{\Re(b_{2n-1})\Re {\left(\frac 1{b_{2n}}\right)}} < \infty$.  Then the following results hold:

\begin{enumerate}
\item If $|t_{2N}-C|\leq C$ then $|t_{2n}(w)-C|\leq C$ for all $n\leq N$.
\item If $\sum_{n=1}^{\infty} |b_n|< \infty$ then $|t_{2n}(w)-C|\leq C$ whenever $|t_{\textrm{even}}(w)-C|\leq C$
\end{enumerate}

\end{cor}

Corollaries \ref{main_theorem1}, \ref{main_theorem2} and \ref{main_theorem3} yield in turn the following results regarding the size of even convergents, odd reverse sequences, and even tail sequences of a Van Vleck fraction $\textnormal{\large K}_{n=1}^\infty \big(\frac{1}{b_n}\big)$ with the coefficients $b_n$ now in a bigger sector $S_{\theta}$ with $0\leq \theta < \pi/2$.

\begin{cor}\label{cor_to_main_theorem}
Let $\textnormal{\large K}_{n=1}^\infty\big(\frac{1}{b_n}\big)$ be a Van Vleck fraction with $b_n \in S_{\theta}$ where $0\leq \theta < \pi/2$.
Then the following results hold:

\begin{enumerate}
\item Assume that \  $\sup\limits_n \left|\frac{b_{2n}}{b_{2n-1}}\right| < \infty$ and let $C^2\geq \frac14 \sup\limits_n \left|\frac{b_{2n}}{b_{2n-1}}\right| \cdot \frac 1{{\cos}^2 \theta}$. If $w\in S_{\theta}$,
$|w-C|\leq C$ then $f_{\textrm{even}}(w) \in S_{\theta}$ and $|f_{\textrm{even}}(w)-C|\leq C$.
\item If $\sum_{n=1}^{\infty} |b_n|= \infty$,  then $|f(0)|^2\leq \sup\limits_n \left|\frac{b_{2n}}{b_{2n-1}}\right|\cdot \frac1{\cos^2\theta}$. 
\item Assume that $\sum_{n=1}^{\infty} |b_n|<\infty$ and  $\sup\limits_n \left|\frac{b_{2n-1}}{b_{2n}}\right| < \infty$ and let $C^2 \geq \frac14\sup\limits_n \left|\frac{b_{2n-1}}{b_{2n}}\right|\cdot \frac1{\cos^2\theta}$. If $w\in S_{\theta}$, $|w-C|\leq C$ then $r_{\textrm{odd}}(w) \in S_{\theta}$ and $|r_{\textrm{odd}}(w)-C|\leq C$.
\item Assume that $\sum_{n=1}^{\infty} |b_n|<\infty$ and $\sup\limits_n \left|\frac{b_{2n}}{b_{2n-1}}\right| < \infty$ and let $C^2 \geq \frac14\sup\limits_n \left|\frac{b_{2n}}{b_{2n-1}}\right|\cdot \frac1{\cos^2\theta}$. If $t_{{\textrm{even}}}(w)\in S_{\theta}$ and $|t_{{\textrm{even}}}(w)-C|\leq C$ then $t_{2n}(w) \in S_{\theta}$ and $|t_{2n}(w)-C|\leq C$ for all $n$.
\end{enumerate}
\end{cor}

Finally, we indicate how the theory described above can be used to control solutions of discrete Dirac equations (\ref{discrete_dirac_matrix_equation}). We have the following useful result.
\begin{cor}\label{dirac_cor}
Let $g_n$, $h_n$ be the solutions of equation (\ref{discrete_dirac_matrix_equation}) with initial condition $g_0=0$, $h_0=1$. Then for every $n$ (and $m$):
\begin{equation*}
\frac{g_n}{h_{n+1}}\leq\sqrt{\sup\limits_l \frac{\beta_{l}}{\alpha_{l+1}}}.
\end{equation*}
\end{cor}

%%%%%%%%%%%%%%%%%%%%%%%%%%%%%%%%%%%%%%%%%%%%%%%%%%%%%%%%%%%%%%%%%%%%%%%%%%%%%%%%%%%%%%%%%%%%%%%%%%%%%%%%%%%%%%%%%%%%%%%%%%%%%%%%%%%%%%%%%%%%%%%%%%%%%%%%%%%%%%%%%%%%%%%%%%%%%%%%

\section{Proofs of Main Theorems}\label{sec3}

The proof of Theorem \ref{main_theorem} relies completely on the lemmas below. They describe the behavior of M\"obius transformations from two-step continued fractions. First we remark that the sector $S_{\theta}$ is closed under addition and reciprocation, i.e. for any $z, w \in S_{\theta}$ we also have $\frac 1z, z+w \in S_{\theta}$.

\begin{lem}\label{key_lemma}
Suppose $p, q$  and $z\in S_{\theta}$ with $p,q \neq 0$ and $0\leq \theta < \pi/4$. If $|z|\leq C$ for some constant $C$ then 
\begin{equation*}
\cfrac{1}{\left|q+\cfrac{1}{p+z}\right|}\leq\cfrac{1}{\left||q|e^{i\theta}+\cfrac{e^{-i\theta}}{|p|+C}\right|}.
\end{equation*}
\end{lem}

\begin{proof}
First we note that for $p,q\in \mathbb C$ we have $\Re{(pq)}+\Re{(p\overline q)}=2\Re (p)\Re (q)$. Since any $z\in S_{\theta}$ satisfies $|z|\cos \theta \leq \Re(z) \leq |z|$ it follows that
\begin{equation*}
\Re{(pq)}\ge 2|p||q|\cos^2\theta - |p||q|= |p||q| \cos 2 \theta.
\end{equation*}

Thus,
\begin{equation*}
\Re\left(\frac{q}{{\bar p+\bar z}}\right)\geq \frac{|q| \cos 2\theta}{|p+z|}\geq  \frac{|q| \cos 2\theta}{|p|+C}.
\end{equation*}
Now we see that 

\begin{equation*}
\begin{aligned}
{\left||q|e^{i\theta}+\cfrac{e^{-i\theta}}{|p|+C}\right|}^2 &= |q|^2 + \frac{|q|\cos 2 \theta}{|p|+C} + \frac 1{(|p|+C)^2}\\
& \leq |q|^2 + \Re\left(\frac{q}{{\bar p+ \bar z}}\right) + \frac 1{|p+z|^2}\\
&={\left|q+\cfrac{1}{p+z}\right|}^2 
 \end{aligned}
\end{equation*}
from which the result follows.
\end{proof}

Using this lemma we will solve the following value region problem, namely we will show that  for large enough $C>0$ the map $\frac 1{q+\frac 1{p+z}}$ preserves the circle centered at the origin with radius $C$, given $p,q,z \in S_{\theta}$ with $0\leq \theta < \pi/4$.

\begin{lem}\label{key_lemma1}
Let $p,q,z, \theta$ be as in Lemma \ref{key_lemma} and $C$ be a constant such that $C^2 \geq \frac{|p|}{|q|\cos 2\theta}$. If $|z|\leq C$ then $\frac{1}{\left|q+\frac{1}{p+z}\right|} \leq C$.
\end{lem}

\begin{proof}

Since $|q|\left(1-\cos^2 2\theta\right)\geq 0$ for $0 \leq \theta < \pi/4$ we have that 
\begin{equation*}
|q|+\frac{2\cos2\theta}{|p|+C}\geq |q|\cos^2 2\theta+\frac{2\cos2\theta}{|p|+C}.
\end{equation*}
Therefore,
\begin{equation*}
\begin{aligned}
{\left||q|e^{i\theta}+\cfrac{e^{-i\theta}}{|p|+C}\right|}^2 &= |q|\left(|q|+ \frac {2\cos 2 \theta}{|p|+C}\right)+ \frac1{(|p|+C)^2}\\
&\geq |q|\left( |q|\cos^2 2\theta+\frac{2\cos2\theta}{|p|+C}\right)+ \frac1{(|p|+C)^2}\\
&= {\left(\cos2\theta|q|+\cfrac{1}{\left(|p|+C\right)}\right)}^2 
\end{aligned}
\end{equation*}
showing $\left||q|e^{i\theta}+\cfrac{e^{-i\theta}}{|p|+C}\right|\geq \cos2\theta|q|+\cfrac{1}{|p|+C}$ . 

Hence,
\begin{equation*}
\begin{aligned}
\left|q+\cfrac{1}{p+z}\right|  &\geq \left| |q|e^{i \theta} + \cfrac {e^{-i\theta}}{|p|+C}\right| \geq  \left||q|\cos2\theta+\cfrac{1}{|p|+C}\right|\\
&\geq |q|\cos2\theta+\cfrac{1}{C^2|q| \cos 2 \theta+C}  = \frac {C^2 |q|^2 \cos^2 2 \theta +C|q| \cos 2 \theta +1}{C (1+C|q|\cos2\theta)}\\
& \geq \frac {C|q| \cos 2 \theta +1}{C (1+C|q|\cos2\theta)} =\frac1C .
\end{aligned}
\end{equation*}
\end{proof}

Now we can proceed with the proof of Theorem \ref {main_theorem}.

\begin{proof}[\bf Proof of Theorem \ref {main_theorem}:] \ 
Since $S_{\theta}$ is closed under addition and reciprocation we can now apply the previous lemma to obtain
\begin{equation*}
\cfrac{1}{\left|q_n+\cfrac{1}{p_n+z}\right|}\leq C.
\end{equation*} 
By $n$ successive applications of the above step we get the desired result.
\end{proof}

The following example illustrates that the angle $\theta =\pi/ 4$ in Theorem \ref{main_theorem} is optimal.

\example
Let $p=q=t e^{i\pi/4}$  and $z= \frac 1{t^{1/3}}e^{\frac {i\pi}4}$ where $0<t\leq 1/2$. We will show that $\frac{1}{\left|q+\frac{1}{p+z}\right|}\geq |z|$. Since $|z|$ can be arbitrarily large for $\theta =\pi/4$ this will then show that statement of Lemma \ref{key_lemma1} does not hold for this case.

First we compute that

\begin{equation*}
{\left|q+ \frac 1{p+z}\right|}^2= {\left|ti+\cfrac 1{t+ \cfrac 1{t^{1/3}}}\right|}^2= t^2 + \frac {t^{2/3}}{{\left(t^{4/3}+1\right)}^2}.
\end{equation*}
Let $s=t^{4/3}$. Since $0< s \leq 1/2$ we see that $s(s^2+s-1)<s(1/4+1/2-1)<0$.  Thus, 

\begin{equation*}
\left|q+\frac{1}{p+z}\right|^2 = t^2 +\frac {t^{2/3}}{{\left(t^{4/3}+1\right)}^2} = s^{3/2}+ \frac{s^{1/2}}{(s+1)^2}= s^{1/2} \left[\frac{s(s^2+s-1)}{(s+1)^2} +1\right] < t^{2/3} = \frac 1{|z|^2}.
\end{equation*}

In order to prove Theorem \ref{main_lemma2} we need the following key lemma on M\"obius transformations.

\begin{lem}
Suppose $p,q,z \in S_{\theta}$ where $0\leq \theta < \pi/2$ and $\Re(q), \Re(p), \Re(z) >0$. Let $C$ be such that $C^2 \geq \frac14 \frac 1{\Re(q) \Re \big({\frac 1p}\big)}$. If $|z-C|\leq C $ then $\left|\frac 1{q+\frac 1{p+z}} -C \right| \leq C$.
\end{lem}

\begin{proof}
First, we notice that for any $w\in \mathbb C$ and constant $K>0$
\begin{equation*}
|w-K|\leq K \;\; \textrm{ if and only if } \;\; \Re \left(\frac 1w \right) \geq \frac 1{2K}.
\end{equation*}
Thus, it suffices to show that $\Re(q)+ \Re \left(\frac 1{p+z}\right) \geq \frac 1{2C}$ whenever $|z-C|\leq C $.  

Since the condition on $C$ implies that $4C^2\Re(q) \Re \left({\frac 1p}\right) \geq 1 >  1-2C \Re(q)$, we have $\Re \left({\frac 1p}\right) \geq \frac {1-2C \Re(q)}{4c^2 \Re(q)}$. Thus,
\begin{equation*}
 \Re \left({\frac 1p}\right) \geq  \frac {1-2C \Re(q)}{4c^2 \Re(q)} =\cfrac 1{2\left(\cfrac {C}{1-2C \Re(q)}-C\right)}= \frac 1{2(B-C)}
\end{equation*}
where $B=\frac C{1-2C \Re(q)}$, which is equivalent to, $|p-(B-C)|\leq B-C$. Thus,

\begin{equation*}
|p+z-B| \leq |z-C|+|p-(B-C)|\leq C+ B-C =B,
\end{equation*}
which then is equivalent to
\begin{equation*}
\Re \left(\frac 1{p+z}\right)\geq \frac 1{2B}=\frac 1{2C}-\Re(q)
\end{equation*}
from which the result follows.
\end{proof}

\begin{proof}[\bf Proof of Theorem \ref{main_lemma2}:]\ 
Applying the above lemma once we obtain,
\begin{equation*}
\left|\frac 1{q_n+\frac 1{p_n+z}} -C \right|\leq C. 
\end{equation*} 
By recursively applying the lemma $n$ times, we get the result.
\end{proof}

\begin{proof}[\bf Proof of Corollary \ref{dirac_cor}:]\ 
Setting $b_{2n}=m\alpha_n$ and $b_{2n+1}=m\beta_n$ and using the initial conditions $g_0=0$, $h_0=1$, we see that equation (\ref{discrete_dirac_matrix_equation}) becomes equation (\ref{H-equations}) for $H_n=B_n$, and $g_n=B_{2n-1}$, $h_n=B_{2n-2}$. Consequently, the ratio $\frac{g_n}{h_{n+1}}=\frac{B_{2n-1}}{B_{2n}}=r_{2n+1}(0)$ is equal to the value of the odd reverse sequence at zero. To prove the estimate for this ratio it is then enough to use Corollary \ref{main_theorem_i} with $\theta=0$ and constant $C$ such that $C^2 = \sup\limits_n \left|\frac{b_{2n}}{b_{2n-1}}\right|= \sup\limits_l \frac{\beta_{l}}{\alpha_{l+1}}$.

\end{proof}

\end{document}